\documentclass[11 pt]{amsart}
\usepackage{amssymb,amsmath,epsfig,mathrsfs, enumerate}

\usepackage{amssymb,amsmath,epsfig,graphics,mathrsfs}
\usepackage{amssymb,amsmath,epsfig,mathrsfs, enumerate, xparse, mathtools}
\usepackage[pagebackref,colorlinks=true,linkcolor=blue,citecolor=blue]{hyperref}

\usepackage{graphicx}
\usepackage[normalem]{ulem}
\usepackage{fancyhdr}
\usepackage[margin=3cm]{geometry}
\hypersetup{
 colorlinks   = true,
 urlcolor     = blue,
 linkcolor    = blue,
 citecolor   = red ,
 bookmarksopen=true
}

\theoremstyle{plain}
\newtheorem{thrm}{Theorem}[section]
\newtheorem{lemma}[thrm]{Lemma}

\newtheorem{dfn}[thrm]{Definition}

\newtheorem*{hyp}{Hypothesis 2.1}
\usepackage{amsmath}
\usepackage{amsfonts}
\usepackage{amssymb}
\usepackage{amsthm}
\usepackage{epsfig,graphics,mathrsfs}
\usepackage{graphicx}

\usepackage{latexsym} 
\usepackage{amssymb,amsmath}
\usepackage{amsthm}
\usepackage{longtable,booktabs,setspace} 
\usepackage{url}

\usepackage[usenames, dvipsnames]{color}

\begin{document}
% begin top matter
% ***************** macroes needed for this paper ************************

\newcommand{\SL}{\mathcal L^{1,p}( D)}
\newcommand{\Lp}{L^p( Dega)}
\newcommand{\CO}{C^\infty_0( \Omega)}
\newcommand{\Rn}{\mathbb R^n}
\newcommand{\Rm}{\mathbb R^m}
\newcommand{\R}{\mathbb R}
\newcommand{\Om}{\Omega}
\newcommand{\Hn}{\mathbb H^n}
\newcommand{\aB}{\alpha B}
\newcommand{\eps}{\ve}
\newcommand{\BVX}{BV_X(\Omega)}
\newcommand{\p}{\partial}
\newcommand{\IO}{\int_\Omega}
\newcommand{\bG}{\boldsymbol{G}}
\newcommand{\bg}{\mathfrak g}
\newcommand{\bz}{\mathfrak z}
\newcommand{\bv}{\mathfrak v}
\newcommand{\Bux}{\mbox{Box}}
\newcommand{\e}{\ve}
\newcommand{\X}{\mathcal X}
\newcommand{\Y}{\mathcal Y}
\newcommand{\W}{\mathcal W}
\newcommand{\la}{\lambda}
\newcommand{\vf}{\varphi}
\newcommand{\rhh}{|\nabla_H \rho|}
\newcommand{\Ba}{\mathscr{B}_\alpha}
\newcommand{\Za}{Z_\beta}
\newcommand{\ra}{\rho_\beta}
\newcommand{\na}{\nabla_\beta}
\newcommand{\vt}{\vartheta}

\numberwithin{equation}{section}

\newcommand{\RN} {\mathbb{R}^N}
\newcommand{\Sob}{S^{1,p}(\Omega)}
\newcommand{\Dxk}{\frac{\partial}{\partial x_k}}
\newcommand{\Co}{C^\infty_0(\Omega)}
\newcommand{\Je}{J_\ve}
\newcommand{\beq}{\begin{equation}}
\newcommand{\bea}[1]{\begin{array}{#1} }
\newcommand{\eeq}{ \end{equation}}
\newcommand{\ea}{ \end{array}}
\newcommand{\eh}{\ve h}
\newcommand{\Dxi}{\frac{\partial}{\partial x_{i}}}
\newcommand{\Dyi}{\frac{\partial}{\partial y_{i}}}
\newcommand{\Dt}{\frac{\partial}{\partial t}}
\newcommand{\aBa}{(\alpha+1)B}
\newcommand{\GF}{\psi^{1+\frac{1}{2\alpha}}}
\newcommand{\GS}{\psi^{\frac12}}
\newcommand{\HFF}{\frac{\psi}{\rho}}
\newcommand{\HSS}{\frac{\psi}{\rho}}
\newcommand{\HFS}{\rho\psi^{\frac12-\frac{1}{2\alpha}}}
\newcommand{\HSF}{\frac{\psi^{\frac32+\frac{1}{2\alpha}}}{\rho}}
\newcommand{\AF}{\rho}
\newcommand{\AR}{\rho{\psi}^{\frac{1}{2}+\frac{1}{2\alpha}}}
\newcommand{\PF}{\alpha\frac{\psi}{|x|}}
\newcommand{\PS}{\alpha\frac{\psi}{\rho}}
\newcommand{\ds}{\displaystyle}
\newcommand{\Zt}{{\mathcal Z}^{t}}
\newcommand{\XPSI}{2\alpha\psi \begin{pmatrix} \frac{x}{|x|^2}\\ 0 \end{pmatrix} - 2\alpha\frac{{\psi}^2}{\rho^2}\begin{pmatrix} x \\ (\alpha +1)|x|^{-\alpha}y \end{pmatrix}}
\newcommand{\Z}{ \begin{pmatrix} x \\ (\alpha + 1)|x|^{-\alpha}y \end{pmatrix} }
\newcommand{\ZZ}{ \begin{pmatrix} xx^{t} & (\alpha + 1)|x|^{-\alpha}x y^{t}\\
     (\alpha + 1)|x|^{-\alpha}x^{t} y &   (\alpha + 1)^2  |x|^{-2\alpha}yy^{t}\end{pmatrix}}
\newcommand{\norm}[1]{\lVert#1 \rVert}
\newcommand{\ve}{\varepsilon}
\newcommand{\sa}{\langle}
\newcommand{\da}{\rangle}

\title[]{A strong unique continuation result for the Baouendi operator}

\author{Agnid Banerjee}
\address{School of Mathematical and Statistical Sciences\\ Arizona State University}\email[Agnid Banerjee]{agnid.banerjee@asu.edu}

\author{Nicola Garofalo}
\address{School of Mathematical and Statistical Sciences\\ Arizona State University}\email[Nicola Garofalo]{nicola.garofalo@asu.edu}

%
% 
% AMS information
%

\maketitle

%\tableofcontents

\begin{abstract}
We establish a strong unique continuation property for the subelliptic Baouendi operator  under the presence of zero-order perturbations satisfying an almost Hardy-type growth condition. In particular, the admissible class includes both $L^\infty_{\mathrm{loc}}$ and singular potentials. We prove that any solution vanishing to infinite order at a point of the degeneracy manifold of the operator must be identically zero. The result holds extends to variable-coefficient operators with intrinsic Lipschitz regularity. A notable feature of the proof is that it relies exclusively on $L^2$ Carleman estimates combined with the classical Hardy inequality.
\end{abstract}

\section{Introduction}\label{S:intro}

In this note we prove a strong unique continuation property (SUCP) for zero-order perturbations of the following degenerate operator 
\begin{equation}\label{pbeta0}
\Ba u= \Delta_z u + |z|^{2\alpha} \Delta_t u,
\end{equation}
where $m,k\in\mathbb N$, $z\in\mathbb R^m$, $t\in\mathbb R^k$ and $\alpha>0$. This operator arises in several contexts in analysis, mathematical physics and geometry (for an extensive discussion we refer the reader to  \cite[sec. 2]{DGGS}).   

Since $\Ba$ is uniformly elliptic away from the degeneracy manifold $M = \{0\}_z\times \R^k$, classical unique continuation results apply there (see, for instance, \cite{AKS}). Consequently, we restrict our attention to the propagation of zeros from points on $M$. By translation invariance, it suffices to consider the origin $(0,0)\in\mathbb R^{m+k}$.

With $\rho_\alpha$ indicating the intrinsic gauge \eqref{rho} and $B_R=\{(z,t)\in\mathbb R^{m+k}\mid \rho_\alpha(z,t)<R\}$, we consider a solution in $B_R$ to the equation 
\begin{equation}\label{main}
\Ba u =  V(z,t)\ u,
\end{equation}
where on $V$ we assume that, for some fixed $R>0$, there exist $C_0>1$ and $\delta \in (0, 2]$ such that one has for $(z,t)\in B_R$:
\begin{equation}\label{vassump}
|V(z,t)|\leq \frac{C_0}{\rho_\alpha(z,t)^{2-\delta}}.
\end{equation}
We remark explicitly that, when  $\delta = 2$, the assumption \eqref{vassump} is equivalent to $V\in L^\infty(B_R)$.

In this connection, and to put our results in perspective, we note that when $\alpha = 1$ the equation \eqref{main} reads
\begin{equation}\label{mainb}
\Delta_z u + |z|^2 \Delta_t u = V(z,t) u.
\end{equation}
The left-hand side in \eqref{mainb} is the well-known operator of Baouendi \cite{Ba} which, when $m = 2n$, is closely connected to the horizontal Laplacian $\Delta_H$ on the Heisenberg group $\mathbb H^n$, or more in general on a group of Heisenberg type with center $\cong \R^k$. An important negative result of H. Bahouri \cite{Bah} showed that the SUCP, and even the weaker UCP, badly fail in $\mathbb H^n$ for $\Delta_H u = V(z,t) u$ (for some positive results,  
 see \cite{GR} and the references therein). 
 
 As a consequence, even for $V\in L^\infty_{loc}$ the SUCP for \eqref{mainb} is all but obvious. In this respect the existing literature relevant to the present note consists of the following works:

\begin{itemize}
\item In \cite{GarShen}, Shen and the second author treated \eqref{mainb} in the case
$m\ge2$, $k=1$. Using a delicate two-weight $L^p$--$L^q$ Carleman estimate, they proved SUCP for $V\in L^p_{\mathrm{loc}}$, $p>p(m)$, including for the first time the case $V\in L^\infty_{\mathrm{loc}}$.
\item In the recent work \cite{DL}, De Bie and Lan significantly extended the results of \cite{GarShen}, allowing any $\alpha\in\mathbb N$ in $\Ba$. They can treat potentials in $L^\infty_{\mathrm{loc}}(\mathbb R^{m+k})$, but their results do not cover the full range $0<\delta\le2$ in \eqref{vassump}.
\end{itemize}

For an open set $\Om\subset \R^{m+k}$, we indicate with $S^{\ell,p}(\Om)$ the natural $L^p$ Sobolev space of order $\ell$ associated with $\Ba$, see Section \ref{S:carl}. We prove the following result.

\begin{thrm}\label{main1}
Assume \eqref{vassump}, $m\ge3$ and $0<\alpha\le1$. Let $u \in S^{2,2}(B_R)$ solve \eqref{main} in $B_R$. If $u$ vanishes to infinite order at $(0,0)$ in the sense of Definition \ref{v0}, then $u \equiv 0$ in $B_R$. 
\end{thrm}

We also have an interesting generalization of this result to the following class of \emph{variable coefficient} operators  
\begin{equation}\label{La}
\mathscr{L}_\alpha= \sum_{i, j=1}^{m+k}X_i(a_{ij}(z,t) X_j).
\end{equation}
In \eqref{La} the vector fields $\{X_1,...,X_{m+k}\}$ are those defining \eqref{pbeta0} (see \eqref{df} and \eqref{pbeta}), and the coefficients $[a_{ij}(z,t)]$ satisfy the best possible intrinsic ``Lipschitz" conditions introduced in \cite{GV}.

\begin{thrm}\label{T:La}
Assume \eqref{vassump}, $m\ge3$ and $0<\alpha\le1$. Let $u \in S^{2,2}(B_R)$ solve
\begin{equation}\label{Lau}
\mathscr L_\alpha u = V(z,t) u
\end{equation}
in $B_R$, where the coefficient matrix $[a_{ij}]$ satisfies the Hypothesis 2.1 below. If $u$ vanishes to infinite order at $(0,0)$ in the sense of Definition \eqref{v0}, then $u \equiv 0$ in $B_R$. 
\end{thrm}

Having stated our main results, we next comment on their proof and compare them to the cited works \cite{GarShen, DL}. Concerning the former aspect, we mention that in our joint paper \cite{BGM} with Manna we proved a $L^2$ Carleman estimate for the operator $\Ba$ in \eqref{pbeta0}, see Theorem \ref{thm2} below. For the ``variable coefficient" operator  in \eqref{La}, such Carleman result was subsequently extended in \cite[Theor. 3.1]{BMa}, see Theorem \ref{thm3}.

We derive Theorems \ref{main1} and \ref{T:La} from a combination of these $L^2$ Carleman estimates with the classical Hardy inequality in $\Rm$. 
This unexpected ``elementary" nature of our approach - the analysis relies exclusively on $L^2$ techniques - represents the essential novelty of our contribution. 

The fact that such an approach is possible at all is already surprising for bounded potentials, and contrasts with the substantially more elaborate machinery employed in the works cited above.
We mention that the limitation on the dimension $m\ge 3$ and on the range $0<\alpha\le 1$ stems exclusively from the Hardy inequality.  

In closing we mention the recent work \cite{AFL} in which the authors developed a delicate refinement of the frequency-function approach in \cite{G} and proved SUCP for potentials $V\in S^{1,\sigma}_{\mathrm{loc}}$ with $\sigma>\frac Q2$. This excludes bounded and singular potentials satisfying \eqref{vassump} for $0<\delta\le1$, although their results do apply to eigenfunctions of $\Ba$.

A brief description of this note is as follows. In Section \ref{pre} we collect the preliminary material which is needed in the proof of Theorems \ref{main1} and \ref{T:La}. In Section \ref{S:proofs} we prove these results.  
%%%%%%%%%%%%%%%%
\vskip 0.2in

\noindent \textbf{Acknowledgement:} We thank Bernard Helffer for his interest in the results presented in this note.

%%%%%%%%%%%%%%%

\vskip 0.2in

%%%%%%%%%%%%%%%%%%%%%%%%%%%%%%%%%%%%%%%%

\section{Preliminary results}\label{pre}

The study of the SUCP for the Baouendi operator \eqref{pbeta0} was initiated by one of us in \cite{G}, where the frequency-function method  in \cite{GLiumj, GLcpam} was developed in the degenerate setting.
These results provided the first examples of subelliptic operators degenerating on a  $k$-dimensional $M$ whose solutions must have a finite order of vanishing at any point of $M$,  see \cite[Theor. 4.2, 4.3 \& Cor. 4.3]{G}. It is worth mentioning that the monotonicity of the frequency in the cited Theor. 4.2 for \eqref{pbeta0} subsequently played a critical role in the work of Caffarelli, Salsa \& Silvestre on the obstacle problem for $(-\Delta)^s$, see \cite[Theor. 3.1 \& Remark 3.2]{CSS}.

The operator \eqref{pbeta0} has two significant features:
\begin{itemize}
\item[(i)] it is invariant with respect to standard translations $(z,t)\to (z,t +t')$ along the manifold of degeneracy $M$. 
\item[(ii)] it is invariant with respect to the following family of anisotropic dilations 
\begin{equation}\label{dil}
\delta_\la(z,t)=(\la z,\la^{\alpha+1} t),\ \ \ \ \ \ \ \ \la>0,
\end{equation}
in the sense that 
\[
\Ba(f\circ \delta_\la) = \la^2  (\Ba f)\circ \delta_\la.
\]
\end{itemize}

A  function $v$ is $\delta_{\la}$-homogeneous of degree $\kappa$ if
\[
v\circ \delta_\la = \la^\kappa\ v.
\]
The infinitesimal generator of the anisotropic dilations \eqref{dil}  is the vector field
\begin{equation}\label{Z}
Z= \sum_{i=1}^m z_i \partial_{z_i} + (\alpha+1)\sum_{j=1}^k t_j \partial_{t_j}.
\end{equation} 
It is standard to verify that $v$ is homogeneous of degree $\kappa$ with respect to \eqref{dil} if and only if the generalized Euler formula holds
\[
Zv=\kappa\ v.
\]

The homogeneous dimension with respect to the anisotropic dilations \eqref{dil}
is the number 
\begin{equation}\label{Q}
Q= m + (\alpha+1) k.
\end{equation}
This terminology is justified by the scaling property of Lebesgue measure
\[
d(\delta_\lambda(z,t)) = \lambda^Q\,dz\,dt.
\]

The following pseudo-gauge 
\begin{equation}\label{rho}
\rho_\alpha(z,t)=(|z|^{2(\alpha+1)} + (\alpha+1)^2 |t|^2)^{\frac{1}{2(\alpha+1)}},
\end{equation}
with its ``ball" and ``sphere" of radius $r$ centered at the origin,
\[
B_r=\{(z,t)\in\mathbb R^{m+k}:\rho(z,t)<r\},
\qquad
S_r=\{(z,t)\in\mathbb R^{m+k}:\rho(z,t)=r\},
\]
plays an important role in the analysis of $\Ba$. It was proved in \cite[Prop. 2.1]{G} that the fundamental solution of $\Ba$  with pole at the origin is given by the formula
\[
\mathscr E_\alpha(z,t) = \frac{C}{\rho_\alpha(z,t)^{Q-2}},\ \ \ \ \ \ \ \ \ (z,t)\not= (0,0),
\]
where $C = C(m,k,\alpha)>0$ is an explicit constant.

To simplify the notation, whenever $\alpha$ is fixed in the discussion we write $\rho$ in place of $\rho_\alpha$. Since it is clear from \eqref{dil} and \eqref{rho} that $\rho$ is one-homogeneous with respect to \eqref{dil}, the generalized Euler formula gives
\begin{equation}\label{hg}
Z\rho=\rho.
\end{equation}
We also need the angle function $\psi$ introduced in \cite{G}
\begin{equation}\label{psi}
\psi = |X\rho|^2= \frac{|z|^{2\alpha}}{\rho^{2\alpha}}.
\end{equation}
The function $\psi$ vanishes on the characteristic manifold $M=\Rn \times \{0\}$, and clearly satisfies $0\leq \psi \leq 1$. Since $\psi$ is homogeneous of degree zero with respect to \eqref{dil}, one has
\begin{equation}\label{Zpsi}
 Z\psi = 0.
 \end{equation}

\medskip

%%%%%%%%%%%%%

Throughout the paper, whenever convenient we will use the summation convention over repeated indices. Let $N = m + k$, and denote an arbitrary point in $\RN$ as  $(z,t) \in \R^m \times \R^k$. If we consider the vector fields $X_1,...,X_N$ in $\R^N$ defined by 
\begin{align}\label{df}
& X_i= \partial_{z_i},\ \ \  i=1, ...m,\ \ \ \  \ \ \ \ X_{m+j}= |z|^{\alpha} \partial_{t_j},\ \ \   j=1, ...k,
\end{align}
then it is immediate to recognise that
\begin{equation}\label{pbeta}
\Ba u= \sum_{i=1}^N X_i^2 u.
\end{equation}
Given a function $f$, we respectively denote by
\begin{equation}\label{grad}
Xf= (X_1f,...,X_Nf),\ \ \ \ \ \  \ |Xf|^2= \sa Xf,Xf\da = \sum_{i=1}^N (X_i f)^2 = |\nabla_z f|^2 + |z|^{2\alpha}|\nabla_t f|^2,
\end{equation}
the intrinsic (degenerate) gradient of a function $f$, and the square of its length. 

If $h\in C^2(\R)$ and $v\in C^2(\R^N)$, then we have the following identities from \cite{G}:
 \begin{equation}\label{ii}
 \Ba h(\rho) = \psi \left(h''(\rho) + \frac{Q-1}{\rho} h'(\rho)\right),
 \end{equation}
 and 
 \begin{equation}\label{h10}
\sa Xv,X\rho\da = \sum_{i=1}^N X_i v X_i \rho = \frac{Zv}{\rho} \psi.
\end{equation}
For the unfamiliar reader, we mention that \eqref{h10} represents the important dividing line between the operator  \eqref{pbeta0} and the above mentioned horizontal Laplacian on the Heisenberg group $\mathbb H^n$. For the latter, the complex geometry causes \eqref{h10} to badly fail. 

Next, we specify the assumptions on the variable-coefficient operator
$\mathscr{L}_\alpha$ in \eqref{La}. 
These hypotheses, introduced in \cite{GV}, encode an intrinsic Lipschitz regularity of the coefficients
$a_{ij}$ and, in the limit $\alpha\to 0^+$, reduce to the classical Lipschitz condition. As is well known, this level of regularity is optimal for unique continuation for second order uniformly elliptic equations, see \cite{AKS, GLcpam}.

\begin{hyp}\label{H}
There exists $\Lambda>0$ such that following estimates hold:
\begin{equation*}
|b_{ij}| := |a_{ij} - \delta_{ij}|\ \leq\
\begin{cases}
\Lambda\rho, \hskip1.4truein \text{ for }\ 1\leq i,\, j\leq m,
\\
\Lambda \psi^{\frac12+\frac1{2\alpha}}\rho\, =\, \Lambda\frac{|z|^{\alpha+1}}{\rho^\alpha},\quad\ \ \text{otherwise},
\end{cases}
\end{equation*}
and
\begin{equation*}
|X_kb_{ij}| = |X_ka_{ij}|\ \leq\
\begin{cases}
\Lambda, \hskip1.1truein \text{ for }\quad  1\leq k\leq m ,\ \text{ and }\ 1\leq i,\,j\leq m
\\
\Lambda \psi^{1+\frac{1}{2\alpha}}\quad \ \ \ \ \ \ \ \ \ \ \ \ \text{when $k >m$ and $\max\{i, j\} > m$}
\\
\Lambda \psi^{1/2}\quad\ \ \ \ \ \ \ \ \ \ \ \ \ \ \text{otherwise}.
	\end{cases}
	\end{equation*}
\end{hyp}

We close this section by introducing the relevant notion of vanishing to infinite order.

\begin{dfn}\label{v0}
  We  say  that $u\in L^2_{loc}(\R^{N})$ vanishes to infinite order at the origin if for every $\ell>0$ one has as $r \rightarrow 0$,
 \begin{align}\label{vanLp}
 \int_{B_r} |u|^2  = O(r^\ell).
 \end{align}
 \end{dfn}
 
For the reader's convenience, and in view of its use in the proofs of Theorems~\ref{main1} and~\ref{T:La}, we record the following extension of a standard result.

\begin{lemma}\label{L:vanish}
Assume \eqref{vassump}, and let \(u \in S^{2,2}(B_R)\) be a weak solution of \eqref{main} in \(B_R\). If \(u\) vanishes to infinite order at \((0,0)\), then the same is true for \(Xu\).
\end{lemma}

\begin{proof}
Fix \(r_0>0\). We use the test function \(w = u \phi^2\) in the weak formulation of \eqref{main}, where
\(\phi(z,t) = h(\rho(z,t))\) is a smooth cutoff function such that
\(\phi \equiv 1\) in \(\{ r_0 < \rho < 2r_0 \}\) and \(\phi \equiv 0\) outside
\(\{ r_0/2 < \rho < 4r_0 \}\).
Integrating by parts yields
\begin{equation}\label{kj1}
\int |Xu|^2 \phi^2
\le 2 \int |Xu|\,|u|\,|\phi|\,|X\phi|
+ \int |V|\,u^2 \phi^2 .
\end{equation}
Using the inequality
\[
2 \int |Xu|\,|u|\,|\phi|\,|X\phi|
\le \frac12 \int |Xu|^2 \phi^2
+ 4 \int |u|^2 |X\phi|^2,
\]
together with the bound on \(V\) in \eqref{vassump} and the estimate
\(|X\phi| \le C/r_0\), we obtain, for a possibly different constant \(C>0\),
\begin{equation}\label{kj2}
\int_{r_0 < \rho < 2r_0} |Xu|^2
\le \frac{C}{r_0^2} \int_{B_{4r_0}} u^2 .
\end{equation}

Since \(u\) vanishes to infinite order at the origin, for every \(k \in \mathbb N\) there exists \(\widetilde C_k>0\) such that
\[
\int_{B_{4r_0}} u^2 \le \widetilde C_k\, r_0^{k+2}.
\]
Substituting this bound into \eqref{kj2} with \(r_0 = 2^{-i} r\) and summing over \(i \in \mathbb N\), we obtain
\begin{align}\label{kj4}
\int_{B_r} |Xu|^2
&= \sum_{i=1}^{\infty} \int_{2^{-i} r < \rho < 2^{-i+1} r} |Xu|^2 \\
&\le C \widetilde C_k\, r^k \sum_{i=1}^{\infty} 2^{-ik}
= C_k\, r^k .
\nonumber
\end{align}
This shows that \(Xu\) vanishes to infinite order at \((0,0)\), completing the proof.

\end{proof}

\medskip

\subsection{The Carleman estimates}\label{S:carl}
We now state the $L^2$ Carleman estimate derived in \cite[Theor. 1.1]{BGM}. As we have mentioned in Section \ref{S:intro}, such estimate plays a crucial role in the proof of our main results Theorem \ref{main1} and \ref{T:La}.

The relevant function spaces for our work are defined as follows. For a multi-index $\beta =(\beta_1,...,\beta_N)$, we routinely denote $X^\beta f = X_1^{\beta_1}...X_N^{\beta_N} f$.  
Given an open set $\Om\subset \R^N$,  we denote by
\[
W^{k,p}(\Om) = \{f\in L^p(\Om)\mid X^\beta f\in L^p(\Om),\ |\beta|\le k\}
\]
the Banach space endowed with the usual norm 
\[
||f||_{W^{k,p}(\Om)} = ||f||_{L^p(\Om)} + \sum_{|\alpha|\le k} ||X^\alpha f||_{L^p(\Om)}.
\]
We also set
\[
S^{k,p}(\Om) = \overline{C^\infty(\Om) \cap W^{k,p}(\Om)}^{W^{k,p}(\Om)},
\]
and 
\[
S^{k,p}_0(\Om) = \overline{C^\infty_0(\Om)}^{W^{k,p}(\Om)}.
\]

Henceforth, to simplify the notation we will routinely omit the relevant measure (volume or surface) in all integrals involved, with the exception of Lemma \ref{hardy}.

 \begin{thrm} \label{thm2}
For every $\ve>0$, there exists $C=C(m, k, \ve)>0$ such that for every $\gamma > \max\{0,(Q-4)/2\}$, $R>0$, and $v \in S^{2,2}_{0}(B_{R_0} \setminus \{0\})$ with  $\operatorname{supp} u \subset B_{R_0} \setminus \{0\}$, one has
\begin{align}\label{est1}
&\gamma^2 \int_{B_{R_0}} \rho^{-2\gamma-4+\ve}v^2 \psi + \int_{B_{R_0}}  \rho^{-2\gamma-2+\ve} |Xv|^2  \leq C R_0^{\ve}  \int_{B_{R_0}} \rho^{-2\gamma}(\Ba v)^2 \psi^{-1}.
\end{align}
 \end{thrm}

We next state the Carleman estimate that we will use in the proof of Theorem \ref{T:La}, see \cite[Theor. 3.1]{BMa}.

\begin{thrm} \label{thm3}
Assume the Hypothesis 2.1 on the matrix of the coefficients. For every $\ve\in (0,1)$, there exist sufficiently large $C=C(\ve)>0$ and $\gamma_0 = \gamma_0(\ve)>0$, and a sufficiently small $R_0 = R_0(\ve)>0$, such that for every $\gamma > \gamma_0$, $R\le R_0$, and $v \in S^{2,2}_{0}(B_{R} \setminus \{0\})$ with  $\operatorname{supp} v \subset B_{R} \setminus \{0\}$, one has
\begin{align}\label{est12}
&\gamma^3 \int \rho^{-2\gamma-4+\ve} e^{2\gamma \rho^\ve} v^2 \psi + \gamma \int  \rho^{-2\gamma-2+\ve} e^{2\gamma \rho^\ve}  \sa A Xv,Xv\da  \leq C  \int \rho^{-2\gamma} e^{2\gamma \rho^\ve} (\mathscr L_\alpha v)^2 \psi^{-1}.
\end{align}
 \end{thrm}
 
For the sake of the reader's understanding, we mention that the additional factor  $e^{2\gamma \rho^{\ve}} $ in the Carleman weight in \eqref{est12} is required to control certain delicate error terms arising in the computations. These terms originate from the ``Lipschitz" perturbation of the principal part of $\mathscr L_\alpha$ and are therefore absent in the ``constant-coefficient" case of
$\Ba$.  Moreover, this exponential factor yields improved coercivity compared to \eqref{est1}: specifically, the estimate features a factor
$\gamma^3$ in front of the zero order term and a factor $\gamma$ in front of the gradient term as opposed to \eqref{est1} above.

We emphasize that the appearance of the two weights $\psi$ and $\psi^{-1}$ on the left- and right-hand sides of $L^2$ estimates such as \eqref{est1} and \eqref{est12} is intrinsic and cannot be avoided. This feature is a major source of analytical difficulty and has, until now, obstructed the derivation with $L^2$ methods of the strong unique continuation property for degenerate operators such as \eqref{pbeta0} and potentials $V\in L^\infty_{loc}$.

Remarkably, we have been able to overcome this difficulty by exploiting the classical Hardy inequality, which we now recall. It is precisely this argument that necessitates the restrictions $m\ge 3$ and $0<\alpha\le 1$.

\begin{lemma}\label{hardy}
Let $m\ge 3$. For $f \in C^{\infty}_0( \R^m)$ we have
\begin{equation*}
\int_{\R^m} \frac{f^2}{|z|^2} dz \leq \left(\frac{2}{m-2}\right)^2 \int_{\R^m} |\nabla_z f|^2 dz. 
\end{equation*}
\end{lemma}

\vskip 0.2in

  %%%%%%%%%%%%%%%%%%%%%%%%%%%%%%%%%%%%%%%%%%%%%%%%%%%%%%%%%% 

 \section{Proof of Theorems \ref{main1} and \ref{T:La}}\label{S:proofs}

In this section we prove our main results. We begin with the

 \begin{proof}[Proof of Theorem \ref{main1}]
 Let $R, \delta >0$ be as in the hypothesis \eqref{main} and \eqref{vassump}, and $R_0<R$ to be chosen later. 
 Denote by $h(s)$ a function in $C^\infty[0,\infty)$, such that $h(s) = 1$ on $0\le s \le R_0/2$ and $h(s) = 0$ for $s\ge R_0$, and denote by $\phi(z,t) = h(\rho(z,t))$ a cut-off function such that 
 $\phi \equiv 1$ in $B_{R_0/2}$ and $\phi \equiv 0$ outside $B_{R_0}$. Using the chain rule and \eqref{ii} we obtain 
\[
X \phi = h'(\rho) X\rho,\ \ \ \ \ \Ba \phi = \psi \left(h''(\rho) + \frac{Q-1}{\rho} h'(\rho)\right).
\]
Keeping \eqref{psi} into account, and that $h'(\rho)$ and $h''(\rho)$ do not vanish only in the region $B_{R_0} \setminus B_{R_0/2}$, we  obtain in such region
\begin{equation}\label{GG}
|X\phi|^2 \leq \frac{C}{R_0^{2}} \psi,\ \ \ \ \ \ \quad (\Ba \phi)^2 \leq \frac{C}{R_0^4}\psi^2,
\end{equation}
for some universal constant $C>0$.

   Since $u$ vanishes to infinite order in the sense of Definition \ref{v0}, and so does $Xu$ by Lemma \ref{L:vanish}, by a standard limiting argument we can apply to $v=u \phi$ the Carleman estimate \eqref{est1}, in which we take $\ve = \delta$, obtaining
\begin{align}\label{e1}
 &\gamma^2 \int \rho^{-2\gamma-4+\delta}v^2 \psi + \int_{B_{R_0}}  \rho^{-2\gamma-2+\delta} |Xv|^2  \leq CR_0^\delta \int \rho^{-2\gamma}(\Ba v)^2 \psi^{-1}. 
 \end{align} 
 
Using the equation \eqref{main} satisfied by $u$, we find
\[
\Ba v = \phi \Ba u + u \Ba \phi + 2 \sa X\phi,Xu\da = V \phi u + u \Ba \phi + 2 \sa X\phi,Xu\da.
\]
Combining this identity with the hypothesis \eqref{vassump} on the zero order term $V$, and with the bounds \eqref{GG}, we find 
\begin{equation}\label{GGG}
(\Ba v)^2 \le C \left[\rho^{2\delta-4} v^2 +  R_0^{-4} u^2 \psi^2 +  R_0^{-4} |Xu|^2 \psi\right], 
\end{equation}
for some universal $C>0$ depending also on $C_0$ in \eqref{vassump}.

If we now use the critical bound \eqref{GGG} to estimate from above the right-hand side in \eqref{e1}, we find 
\begin{align}\label{u1}
 & \gamma^2 \int \rho^{-2\gamma-4+\delta}v^2 \psi + \int_{B_{R_0}}  \rho^{-2\gamma-2+\delta} |Xv|^2  
\\
& \leq CR_0^\delta \int_{B_{R_0}} \rho^{-2\gamma-4 +2\delta} v^2  \psi^{-1} +   C 4^{\gamma} R_0^\delta R_0^{-2 \gamma-4} \int_{B_{R_0} \setminus B_{R_0/2}} \left(|Xu|^2  + u^2 \psi \right). 
\notag
\end{align}

At this stage it is important to highlight the main obstruction in \eqref{u1}, namely the first integral on the right-hand side. Ideally, one would like to absorb this term into the left-hand side; however, this is prevented by the presence of the singular weight $\psi^{-1}$, which directly opposes the weight $\psi$ appearing in the coercive term on the left-hand side of \eqref{u1}. 

In what follows, we show how this difficulty can be overcome in the regime $0<\alpha\le 1$ in \eqref{pbeta0}. Under this hypothesis,  
we obtain from the definition of $\psi$ in \eqref{psi} 
\begin{equation}\label{ps1}
 \psi^{-1} = \frac{\rho^{2\alpha}}{|z|^{2\alpha}}\leq \frac{\rho^2}{|z|^2}.
 \end{equation}
Using \eqref{ps1} we can bound from above the first term on the right-hand side of \eqref{u1} by  
\[
CR_0^\delta \int \rho^{-2\gamma -2+2\delta}\frac{v^2}{|z|^2}.
\]

At this point we write this integral as an iterated integral and, to that on $\R^m$, we apply Lemma \ref{hardy} to the function 
\[
f = \rho^{-\gamma -1+\delta} v
\]
(this is possible since, by the assumptions on $u$, for a.e. $t \in \mathbb R^k$ with $|t|>0$ the function $z\to f(z,t)$ is in the standard Sobolev space $W^{1,2}(\Rm)$). Precisely, we proceed as follows
\begin{align}\label{j1}
 & CR_0^\delta \int \rho^{-2\gamma -2+2\delta}\frac{v^2}{|z|^2}=  CR_0^\delta \int_{\R^k} \int_{\R^m} \frac{(\rho^{-\gamma-1+\delta}v)^2}{|z|^2} \leq C_1 R_0^\delta \int_{\R^k} \int_{\R^m} |\nabla_z(\rho^{-\gamma-1+\delta}v)|^2 
\\
& \leq C_1 R_0^\delta \int \rho^{-2\gamma -2 +2\delta} |\nabla_z v|^2 + C_1  R_0^\delta \gamma^2 \int \rho^{-2\gamma -4 +2\delta} v^2 |\nabla_z \rho|^2. 
\notag
\end{align}
 
By the second identity in \eqref{grad}, we now make the trivial observation that
\[
|\nabla_z v|^2 \leq |Xv|^2,\ \ \ \ \ |\nabla_z \rho|^2 \leq |X\rho|^2 = \psi.
\]
Therefore, for $R_0$ sufficiently small the first term on the right-hand side of \eqref{j1} can be absorbed  by the term $\int_{B_{R_0}}  \rho^{-2\gamma-2+\delta} |Xv|^2$ in \eqref{e1}. Similarly, the second term on the right-hand side in \eqref{j1} can be absorbed by the term $\gamma^2 \int \rho^{-2\gamma-4+\delta}v^2 \psi $ in \eqref{e1} provided $R_0$ is sufficiently small. 

In conclusion, we obtain from \eqref{u1}  \begin{align}\label{e2}
  &\gamma^2 \int \rho^{-2\gamma-4+\delta}v^2 \psi  \leq  C 4^{\gamma} R_0^\delta R_0^{-2 \gamma-4} \int_{B_{R_0} \setminus B_{R_0/2}} \left(|Xu|^2  + u^2 \psi \right). 
\end{align}
Since $\phi \equiv 1 $ in $B_{2^{-1}R_0}$, the integral on the left-hand side of \eqref{e2} can be bounded from below as follows
\begin{equation}\label{l1}
\gamma^2 \int \rho^{-2\gamma-4+\delta}v^2 \psi  \geq \gamma^2 R_0^{-4 \gamma -8 + 2\delta} \int_{B_{R_0^2}} u^2 \psi.
\end{equation}
  Here we used the fact that for all $R_0$ small, $R_0^2 < 2^{-1}R_0$ and therefore $\phi \equiv 1$ in $B_{R_0^2}$. 
  Using this in \eqref{e2} we finally obtain
  \begin{equation}\label{con1}
  \alpha^2 4^{-\gamma} R_0^{-2\gamma - 4 +\delta} \int_{B_{R_0^2}} u^2 \psi \leq C \int_{B_{R_0} \setminus B_{2^{-1}R_0}} \left(|Xu|^2  + u^2 \psi\right).  
\end{equation}
  
By further restricting $R_0$ so that $R_0 <1/4$, we can ensure that 
$4^{-\gamma} R_0^{-\gamma} > 1$.
 Using this observation in \eqref{con1}, we finally obtain
\begin{equation}\label{con2}
\gamma^2  R_0^{-\gamma - 4 +\delta} \int_{B_{R_0^2}} u^2 \psi \leq C \int_{B_{R_0} \setminus B_{2^{-1}R_0}} \left(|Xu|^2  + u^2 \psi\right).  
\end{equation}  

By finally letting $\gamma \to \infty$ we conclude  $u \equiv 0$ in $B_{R_0^2}$. One can spread the zero set by the classical theory of weak unique continuation for uniformly elliptic operators with bounded perturbations (see for instance \cite{AKS}, \cite{GLcpam}).
  
\end{proof}

\begin{proof}[Proof of Theorem~\ref{T:La}]
The argument follows closely the proof of Theorem~\ref{main1}. The only essential difference is that, in place of the Carleman estimate in Theorem~\ref{thm2}, we employ the variable-coefficient estimate \eqref{est12} from Theorem~\ref{thm3}. For the reader's convenience, we outline the main steps.

Applying \eqref{est12} with \(\varepsilon=\delta\) to \(v = u \phi\), where \(\phi\) is the cutoff function introduced in the proof of Theorem~\ref{main1}, we obtain
\begin{align}\label{no1}
&\gamma^3 \int \rho^{-2\gamma-4+\delta} e^{2\gamma \rho^\delta} v^2 \psi
+ \gamma \int_{B_{R_0}} \rho^{-2\gamma-2+\delta} e^{2\gamma \rho^\delta}
\langle A Xv, Xv \rangle \\
&\hspace{3cm}\le C \int \rho^{-2\gamma} e^{2\gamma \rho^\delta}
(\mathscr L_\alpha v)^2 \psi^{-1}.
\nonumber
\end{align}

As in the ``constant-coefficient" case of $\Ba$, using the equation \eqref{Lau} satisfied by $u$ together with the bound on $V$, we obtain the analogue of \eqref{GGG}, namely
\begin{equation}\label{G1}
(\mathscr L_\alpha v)^2
\le C \Bigl(
\rho^{2\delta-4} v^2
+ R_0^{-4} u^2 \psi^2
+ R_0^{-4} |Xu|^2 \psi
\Bigr).
\end{equation}
A detailed derivation of \eqref{G1} can be found in \cite[p.~653]{GV}.

Substituting \eqref{G1} into \eqref{no1} yields
\begin{align}\label{no2}
&\gamma^3 \int \rho^{-2\gamma-4+\delta} e^{2\gamma \rho^\delta} v^2 \psi
+ \gamma \int \rho^{-2\gamma-2+\delta} e^{2\gamma \rho^\delta}
\langle A Xv, Xv \rangle \\
&\le C \int_{B_{R_0}} \rho^{-2\gamma-4+2\delta}
e^{2\gamma \rho^\delta} v^2 \psi^{-1}
+ C\,4^\gamma R_0^{-2\gamma-4} e^{2\gamma R_0^\delta}
\!\!\int_{B_{R_0}\setminus B_{R_0/2}}
\bigl(|Xu|^2 + u^2 \psi\bigr).
\nonumber
\end{align}

Proceeding as in \eqref{j1}, we use \eqref{ps1}, but now apply the Hardy inequality in the $z$-variable to the function
\[
f = \rho^{-\gamma-1+\delta} e^{\gamma \rho^\delta} v .
\]
Noting that
\begin{align*}
|\nabla_z (\rho^{-\gamma-1+\delta} e^{\gamma \rho^\delta})|
&\le e^{\gamma \rho^\delta} |\nabla_z \rho|
\Bigl( (\gamma+1-\delta)\rho^{-\gamma-2+\delta}
+ \gamma\delta \rho^{-\gamma-2+2\delta} \Bigr) \\
&\le C_0 \gamma\, e^{\gamma \rho^\delta}
\rho^{-\gamma-2+\delta} |\nabla_z \rho|,
\end{align*}
we obtain, by a computation analogous to that leading to \eqref{j1},
\begin{align}\label{no4}
C \int_{B_{R_0}} \rho^{-2\gamma-4+2\delta}
e^{2\gamma \rho^\delta} v^2 \psi^{-1}
&\le C_1 \int \rho^{-2\gamma-2+2\delta}
e^{2\gamma \rho^\delta} |\nabla_z v|^2 \\
&\quad + C_1 \gamma^2
\int \rho^{-2\gamma-4+2\delta}
e^{2\gamma \rho^\delta} v^2 |\nabla_z \rho|^2 .
\nonumber
\end{align}

Since
\[
|\nabla_z v|^2 \le |Xv|^2,
\qquad
|\nabla_z \rho|^2 \le |X\rho|^2 = \psi,
\]
we infer that, for \(R_0\) sufficiently small and \(\gamma\) sufficiently large, the first term on the right-hand side of \eqref{no4} can be absorbed by the gradient term on the left-hand side of \eqref{no1}, using the ellipticity condition
\(\langle A Xv, Xv \rangle \approx |Xv|^2\).
Similarly, the second term in \eqref{no4} can be absorbed by the zero-order term
\(\gamma^3 \int \rho^{-2\gamma-4+\delta} e^{2\gamma \rho^\delta} v^2 \psi\)
in \eqref{no1}.

The remainder of the proof proceeds exactly as in
\eqref{e2}-\eqref{con2}, and we omit the repetition.

\end{proof}

%%%%%%%%%%%%%%%%%%%%%%%%%%%%%%%%%%%%%%%%%%%%%%%

\end{document}